\documentclass[11pt]{amsart}
\usepackage{color}
\usepackage[dvips]{graphicx}
\usepackage{amscd}
\usepackage{amsmath}
\usepackage{amssymb}
\usepackage{latexsym}

\newcommand{\tun}{\begin{picture}(5,0)(-2,-1)
\put(0,0){\circle*{2}}
\end{picture}}

\newcommand{\tdeux}{\begin{picture}(7,7)(0,-1)
\put(3,0){\circle*{2}}
\put(3,5){\circle*{2}}
\put(3,0){\line(0,1){5}}
\end{picture}}

\newcommand{\ttroisun}{\begin{picture}(15,12)(-5,-1)
\put(3,0){\circle*{2}}
\put(6,7){\circle*{2}}
\put(0,7){\circle*{2}}
\put(-0.65,0){$\vee$}
\end{picture}}

\newcommand{\ttroisdeux}{\begin{picture}(5,15)(-2,-1)
\put(0,0){\circle*{2}}
\put(0,5){\circle*{2}}
\put(0,10){\circle*{2}}
\put(0,0){\line(0,1){5}}
\put(0,5){\line(0,1){5}}
\end{picture}}

\newcommand{\tdun}[1]{\begin{picture}(10,5)(-2,-1)
\put(0,0){\circle*{2}}
\put(3,-2){\tiny #1}
\end{picture}}

\newcommand{\tddeux}[2]{\begin{picture}(12,10)(0,-1)
\put(3,0){\circle*{2}}
\put(3,5){\circle*{2}}
\put(3,0){\line(0,1){5}}
\put(6,-2){\tiny #1}
\put(6,3){\tiny #2}
\end{picture}}

\newcommand{\tdtroisun}[3]{\begin{picture}(20,12)(-5,-1)
\put(3,0){\circle*{2}}
\put(6,7){\circle*{2}}
\put(0,7){\circle*{2}}
\put(-0.65,0){$\vee$}
\put(5,-2){\tiny #1}
\put(9,5){\tiny #2}
\put(-5,5){\tiny #3}
\end{picture}}

\newcommand{\tdtroisdeux}[3]{\begin{picture}(12,15)(-2,-1)
\put(0,0){\circle*{2}}
\put(0,5){\circle*{2}}
\put(0,10){\circle*{2}}
\put(0,0){\line(0,1){5}}
\put(0,5){\line(0,1){5}}
\put(3,-2){\tiny #1}
\put(3,3){\tiny #2}
\put(3,9){\tiny #3}
\end{picture}}

\newcommand{\ptroisun}{\begin{picture}(15,12)(-5,-1)
\put(3,7){\circle*{2}}
\put(-0.65,0){$\wedge$}
\put(6,0){\circle*{2}}
\put(0,0){\circle*{2}}
\end{picture}}

\newcommand{\pquatresix}{\begin{picture}(15,12)(-5,-1)
\put(0,0){\circle*{2}}
\put(7,0){\circle*{2}}
\put(0,7){\circle*{2}}
\put(7,7){\circle*{2}}
\put(0,0){\line(0,1){7}}
\put(7,0){\line(0,1){7}}
\put(0,1.5){$\scriptstyle \diagdown$}
\end{picture}}

\newcommand{\pdtroisun}[3]{\begin{picture}(20,12)(-5,-1)
\put(3,7){\circle*{2}}
\put(-0.65,0){$\wedge$}
\put(6,0){\circle*{2}}
\put(0,0){\circle*{2}}
\put(5,5){\tiny #1}
\put(-5,-2){\tiny #2}
\put(9,-2){\tiny #3}
\end{picture}}

\newcommand{\pdcinq}[5]{\begin{picture}(20,23)(-5,-1)
\put(3,0){\circle*{2}}
\put(-0.65,0){$\vee$}
\put(6,7){\circle*{2}}
\put(0,7){\circle*{2}}
\put(3,14){\circle*{2}}
\put(3,21){\circle*{2}}
\put(3,14){\line(0,1){7}}
\put(-0.65,7){$\wedge$}
\put(5,-2){\tiny #1}
\put(9,5){\tiny #2}
\put(-5,5){\tiny #3}
\put(5,12){\tiny #4}
\put(5,19){\tiny #5}
\end{picture}}

\input{xy}
\xyoption{all}

\newtheorem{defi}{\indent Definition}
\newtheorem{lemma}[defi]{\indent Lemma}
\newtheorem{cor}[defi]{\indent Corollary}
\newtheorem{theo}[defi]{\indent Theorem}
\newtheorem{prop}[defi]{\indent Proposition}

\newcommand{\sym}{\mathfrak{S}}
\newcommand{\VS}{\mathbf{S}}

\newcommand{\QSym}{\mathbf{QSym}}

\newcommand{\T}{\mathcal{T}}
\newcommand{\TT}{\mathbf{T}}
\newcommand{\SF}{\mathbf{F}}
\newcommand{\CF}{\mathcal{F}}

\newcommand{\F}{\mathcal{F}}

\renewcommand{\S}{\mathfrak{S}}

\newcommand{\N}{\mathbb{N}}

\def\shuff#1#2{\mathbin{
      \hbox{\vbox{\hbox{\vrule \hskip#2 \vrule height#1 width 0pt}\hrule}\vbox{\hbox{\vrule \hskip#2 \vrule height#1 width 0pt\vrule }\hrule}}}}
\def\shuffl{{\mathchoice{\shuff{5pt}{3.5pt}}{\shuff{5pt}{3.5pt}}{\shuff{3pt}{2.6pt}}{\shuff{3pt}{2.6pt}}}}
\def\shuffle{\, \shuffl \,}

\def\<{\langle}
\def\>{\rangle}

\begin{document}

\title[Finite spaces]{Infinitesimal and $B_\infty$-algebras, finite spaces, and quasi-symmetric functions}
\date{}

\author{Lo\"\i c Foissy}
\address{F\'ed\'eration de Recherche Math\'ematique du Nord Pas de Calais FR 2956\\
Universit\'e du Littoral C\^ote d'opale\\ 50, rue Ferdinand Buisson, CS 80699\\ 62228 Calais Cedex, France}
\email{foissy@lmpa.univ-littoral.fr}

\author{Claudia Malvenuto}
\address{Dipartimento di Matematica\\ Sapienza Universit\`a
di Roma\\ P.le A. Moro 5\\ 00185, Roma, Italy}
\email{claudia@mat.uniroma1.it}

\author{Fr\'ed\'eric Patras}
\address{UMR 7351 CNRS\\
        		Universit\'e de Nice\\
        		Parc Valrose\\
        		06108 Nice Cedex 02
        		France}
        		\email {patras@unice.fr}

\begin{abstract}

Finite topological spaces are in bijective correspondence with preorders on finite sets. We undertake their study using combinatorial tools that have been developed to investigate general discrete structures. 
A particular emphasis will be put on recent topological and combinatorial Hopf algebra techniques. We will show that the linear span of finite spaces carries generalized Hopf algebraic
structures that are closely connected with familiar constructions and structures in topology (such as the one of 
cogroups in the category of associative algebras
that has appeared e.g. in the study of loop spaces of suspensions).
The most striking results that we obtain are certainly that the linear span of finite spaces carries the structure of the enveloping algebra of a $B_\infty$--algebra, and that there are natural (Hopf algebraic) morphisms between finite spaces and quasi-symmetric functions.
In the process, we introduce the notion of Schur-Weyl categories in order to describe rigidity theorems for cogroups in the category of associative algebras and related structures, as well as to account for the existence of natural operations (graded permutations) on them.
\end{abstract}
\maketitle

\section{Introduction}

Finite topological spaces, or finite spaces, for short, that is, topologies on finite sets, have a long history, going back at least to P.S. Alexandroff \cite{Alexandroff}. He was the first to investigate, in 1937,
 finite spaces from a combinatorial point of view
and relate them to preordered sets. Indeed, finite spaces happen to be in bijective correspondence with preorders on finite sets
and it is extremely tempting to undertake their study using the combinatorial tools that have been developed to investigate general discrete structures. 
However, quite surprisingly, such an  undertaking does not seem to have taken place so far, and it is the purpose of the present article to do so.

A particular emphasis will be put on recent topological and combinatorial Hopf algebra techniques. We will show that the set of finite spaces carries naturally (generalized) Hopf algebraic
structures that are closely connected with usual topological constructions (such as joins or cup products) and familiar structures in topology (such as the one of 
cogroups in the category of associative algebras, or infinitesimal Hopf algebras, 
that have appeared e.g. in the study of loop spaces of suspensions and the Bott-Samelson theorem \cite{Bott,Berstein}). Let us mention that the operation underlying the Hopf algebra coproduct is less standard and amounts to the ``extraction'' of open subsets out of finite spaces (Definition~\ref{def8}).
The most striking results that we obtain are certainly that, first, the linear span $\mathcal F$ of finite spaces carries the structure of the enveloping algebra of a $B_\infty$--algebra (Theorem~\ref{th9}). Second, that there is a (surjective, structure preserving) Hopf algebra morphism from $\mathcal F$ to the algebra of quasi-symmetric functions (Theorem~\ref{th11}). In the process, we introduce the notion of Schur-Weyl categories to describe rigidity theorems for cocommutative cogroups in the category of associative algebras (or, equivalently, infinitesimal bialgebras) and related structures such as shuffle bialgebras or their  dual bialgebras. Here, rigidity has to be understood in the sense of Livernet \cite{Livernet}: a generalized bialgebraic structure, such as a cogroup in the category of associative algebras, is rigid if it is free as an algebra and cofree as a coalgebra.

Let us point out that operations such as cup products are usually defined ``locally'', that is, inside a chain or cochain algebra associated to a \it given \rm topological space, whereas
the structures we introduce hold ``globally'' over the linear span of all finite spaces. 
Although we will not investigate systematically in the present article this interplay between ``local'' and ``global'' constructions, 
it is certainly one of the interesting phenomena showing up in the study of finite topological spaces.

From the historical perspective, a systematic homotopical investigation of finite spaces did not occur till the mid-60's, with breakthrough contributions by R.E. Stong \cite{Stong} and M.C. McCord \cite{McCord1,McCord2}. These investigations were revived in the early 2000s, among others under the influence of \mbox{P. May}; 
we refer to \cite{BarmakThesis} for details. These studies focussed largely on problems such as reduction methods 
(methods to remove points from finite spaces without changing their strong or weak homotopy type and related questions such as the construction of minimal spaces, 
see e.g. \cite{BarmakMinian, Fieux}), as such they are complementary to the ones undertaken in the present article.

The article is organized as follows: in the next two sections, we review briefly the links between finite spaces and preorders, introduce the $Com-As$ structure on finite spaces and study its properties 
(freeness, involutivity, compatibility with homotopy reduction methods).
Sections 4 and 5 revisit the equivalent notions of free algebras and cofree coalgebras, cocommutative cogroups in the category of associative algebras and infinitesimal bialgebras \cite{Berstein,Loday,Livernet}. We extend in particular results of Livernet and relate these algebras to shuffle bialgebras and their dual bialgebras.
In the following section, we show how these ideas apply to finite spaces, showing in particular that their linear span carries the structure of a cofree connected coalgebra and, more precisely, is the enveloping algebra of a $B_\infty$--algebra.
The last section investigates the links between finite topologies and quasi-symmetric functions.

In the present article, we study ``abstract'' finite spaces, that is, finite spaces up to homeomorphisms: we identify two topologies $\mathcal T$ and ${\mathcal T}'$ on the finite sets $X$ and $Y$ if there exists a set map $f$ from $X$ to $Y$ inducing an isomorphism between $\mathcal T$ and ${\mathcal T}'$. The study of ``decorated'' finite spaces (that is, without taken into account this identification) is interesting for other purposes (e.g. enumerative and purely combinatorial ones). These questions will be the subject of another article \cite{FoissyMalv}. 

All vector spaces and algebraic structures (algebras, coalgebras...) are defined over a field $K$ of arbitrary characteristic. By linear span of a set $X$, we mean the vector space freely generated by $X$ over this ground field.
Unless otherwise stated, the objects we will consider will always be $\N$-graded (shortly, graded) and connected (connectedness meaning as usual that the degree 0 component of a graded vector space is the null vector space or is the ground field for a graded algebra, coalgebra or bialgebra). Because of this hypothesis, the two notions of Hopf algebras and bialgebras will agree (see e.g. \cite{hazewinkel2010algebras}); we will use them equivalently and without further comments.

\ \\

The authors acknowledge support from the grant CARMA ANR-12-BS01-0017. L. Foissy and F. Patras acknowledge visiting support from Sapienza Universit\`a di Roma.

\section{Topologies on finite sets}

\subsection{Notation and definitions}

Let $X$ be a set. Recall that a topology on $X$ is a family $\T$ of subsets of $X$, called the open sets of $\T$, such that:
\begin{enumerate}
\item $\emptyset$, $X \in \T$.
\item The union of an arbitrary number of elements of $\T$ is in $\T$.
\item The intersection of a finite number of elements of $\T$ is in $\T$.
\end{enumerate}

When $X$ is finite, these axioms simplify: a topology on $X$ is a family of subsets containing the empty set and $X$ and 
closed under unions and intersections. 
In particular, the set of complements of open sets (the closed sets for $\T$, which is automatically closed under unions and intersections) 
defines a dual topology $\T^\ast$ as $\T^\ast:=\{F\subset X, \ \exists O\in\T,\ F=X-O\}$. We will write sometimes $\sigma$ for the duality 
involution, $\sigma(\T):=\T^\ast,\ \sigma^2=Id$.

Two topologies $\T$, $\T'$, on finite sets respectively $X$ and $Y$, are homeomorphic if and only if there exists a bijective map $f$ between $X$ and $Y$ 
such that $f_\ast(\T)=\T'$ (where we write $f_\ast$ for the induced map on subsets of $X$ and $Y$).
We call {\it finite spaces} the equivalence classes of finite set topologies under homeomorphisms and write $\overline\T$ for the 
finite space associated to a given topology $\T$ on a finite set $X$. In order to avoid terminological ambiguities, we will a finite set $X$ equipped with a topology a \it finite topological set \rm (instead of finite topological space).

Every finite space $\overline\T$ can be represented by a (non-unique) topology $\overline\T_n$
on the set  $[n]:=\{1,...,n\}$ (in particular, $[0]=\emptyset$); we call $\overline\T_n$ a standard representation of $\overline\T$. 
The duality involution goes over to finite spaces, its action on finite spaces is still written $\sigma$ (or with a $\ast$).

Let us recall now the bijective correspondence between topologies on a finite set $X$ and preorders on $X$ (see \cite{Erne}).
\begin{enumerate}

\item Let $\T$ be a topology on the finite set $X$. The relation $\leq_\T$ on $X$ is defined by
$i\leq_\T j$ if any open set of $\T$ which contains $i$ also contains $j$. Then $\leq_\T$ is a preorder, that is to say a reflexive, transitive relation.
Moreover, the open sets of $\T$ are the ideals of $\leq_\T$, that is to say the sets $I\subseteq X$ such that, for all $i,j \in X$:
$$(i\in I\mbox{ and } i\leq_\T j)\Longrightarrow j\in I.$$
\item Conversely, if $\leq$ is a preorder on $X$, the ideals of $\leq$ form a topology on $X$ denoted by $\T_\leq$.
Moreover, $\leq_{T_\leq}=\leq$, and $\T_{\leq_\T}=\T$. Hence, there is a bijection between the set of topologies on $X$ and the 
set of preorders on $X$. A map between finite topologies (i.e. topologies on finite sets) is continuous if and only if it is preorder-preserving.

\item Let us define for each point $x\in X$ the set $U_x$ to be the minimal open set containing $x$. The $U_x$ form a basis for the topology of $X$ called the minimal basis of $\T$. The preorder that has just been introduced can be equivalently defined by $x\leq_\T y \Leftrightarrow y\in U_x$. Notice that the opposite convention (defining a preorder from a topology using the requirement $ x\in U_y$) would lead to equivalent results.

\item Let $\T$ be a topology on $X$. The relation $\sim_\T$ on $X$, defined by $i \sim_\T j$ if $i\leq_\T j$ and $j\leq_\T i$,
is an equivalence relation on $X$. Moreover, the set $X/\sim_\T$ is partially ordered by the relation defined on the equivalence classes $\overline{i}$ by
$\overline{i}\leq_\T \overline{j}$ if $i\leq_\T j$. Consequently, we shall represent preorders on $X$ (hence, topologies on $X$)
by the Hasse diagram of $X/\sim_\T$, the vertices being the equivalence classes of $\sim_\T$. 
\item Duality between topologies is reflected by the usual duality of preorders: $i\leq_{\T^\ast}j\Leftrightarrow j\leq_{\T}i$.
In particular, the Hasse diagram of $\T^\ast$ is obtained by reversing (turning upside-down) the Hasse diagram of $\T$.
\item A topological space is $T_0$ if it satisfies the separation axiom according to which the relation $\sim$ is trivial (equivalence classes for $\sim$ are singletons, that is,
for any two points $x,y\in X$, there always exists an open set containing only one of them).
At the level of $\leq_\T$ this amounts requiring antisymmetry: the preorder $\leq_\T$ is then a partial order. 
In other terms, finite $T_0$-spaces are in
bijection with isomorphism classes of finite partially ordered sets (posets).
\end{enumerate}

For example, here are the topologies on $[n]$, $n\leq 3$:
$$1=\emptyset\:; \tdun{$1$}; \tdun{$1$}\tdun{$2$},\tddeux{$1$}{$2$},\tddeux{$2$}{$1$},\tdun{$1,2$}\hspace{3mm};$$
$$\tdun{$1$}\tdun{$2$}\tdun{$3$},\tddeux{$1$}{$2$}\tdun{$3$},\tddeux{$1$}{$3$}\tdun{$2$},
\tddeux{$2$}{$1$}\tdun{$3$},\tddeux{$2$}{$3$}\tdun{$1$},\tddeux{$3$}{$1$}\tdun{$2$},$$
$$\tddeux{$3$}{$2$}\tdun{$1$}, \tdtroisun{$1$}{$3$}{$2$},\tdtroisun{$2$}{$3$}{$1$},\tdtroisun{$3$}{$2$}{$1$},
\pdtroisun{$1$}{$2$}{$3$},\pdtroisun{$2$}{$1$}{$3$},\pdtroisun{$3$}{$1$}{$2$},
\tdtroisdeux{$1$}{$2$}{$3$},\tdtroisdeux{$1$}{$3$}{$2$},\tdtroisdeux{$2$}{$1$}{$3$},
\tdtroisdeux{$2$}{$3$}{$1$},\tdtroisdeux{$3$}{$1$}{$2$},\tdtroisdeux{$3$}{$2$}{$1$},$$
$$\tdun{$1,2$}\hspace{3mm} \tdun{$3$},\tdun{$1,3$}\hspace{3mm} \tdun{$2$},\tdun{$2,3$}\hspace{3mm} \tdun{$1$},
\tddeux{$1,2$}{$3$}\hspace{3mm},\tddeux{$1,3$}{$2$}\hspace{3mm},\tddeux{$2,3$}{$1$}\hspace{3mm},
\tddeux{$3$}{$1,2$}\hspace{3mm},\tddeux{$2$}{$1,3$}\hspace{3mm},\tddeux{$1$}{$2,3$}\hspace{3mm},
\tdun{$1,2,3$}\hspace{5mm}.$$

The two topologies on $[3]$, $\tdtroisun{$1$}{$3$}{$2$}$ and $\pdtroisun{$1$}{$2$}{$3$}$,
are dual. 

A finite space will be represented by an unlabelled Hasse diagram.
The cardinalities of the equivalence classes of $\sim_\T$ are indicated on the diagram associated to $\T$ if they are not equal to $1$.
Here are the finite spaces of cardinality $\leq 3$:
$$1=\emptyset;\tun;\tun\tun,\tdeux,\tdun{$2$};\tun\tun\tun,\tdeux\tun,\ttroisun,\ptroisun,\ttroisdeux,\tdun{$2$}\tun,\tddeux{$2$}{},\tddeux{}{$2$},
\tdun{$3$}.$$
The (minimal) finite space realization, up to weak homotopy equivalence, of the circle  and of the 2-dimensional sphere  (see e.g. \cite{Barmak2007})
\begin{displaymath}
\xymatrix@C=4pt{ \bullet \ar@{-}[d] \ar@{-}[dr]  & \bullet \ar@{-}[d] \ar@{-}[dl] &&&&& \bullet \ar@{-}[d] \ar@{-}[dr]  & \bullet \ar@{-}[d] \ar@{-}[dl] \\ 
		\bullet &  \bullet &&&&&  \bullet \ar@{-}[d] \ar@{-}[dr]  & \bullet \ar@{-}[d] \ar@{-}[dl] \\ 
		&&&&&&\bullet &  \bullet } 
\end{displaymath}
are examples of self-dual finite spaces.

The number $t_n$ of topologies on $[n]$ is given by the sequence A000798 in \cite{Sloane}:
{\small{
$$\begin{array}{c|c|c|c|c|c|c|c|c|c|c}
n&1&2&3&4&5&6&7&8&9&10\\
\hline t_n&1&4&29&355&6\:942&209527&9\:535\:241&642\:779\:354&63\:260\:289\:423
&8\:977\:053\:873\:043
\end{array}$$}}
The set of topologies on $[n]$ will be denoted by $\TT_n$, and we put $\displaystyle \TT=\bigsqcup_{n\geq 0} \TT_n$. \\

The number $f_n$ of finite spaces with $n$ elements is given by the sequence A001930 in \cite{Sloane}:
$$\begin{array}{c|c|c|c|c|c|c|c|c|c|c}
n&1&2&3&4&5&6&7&8&9&10\\
\hline f_n&1&3&9&33&139&718&4\:535& 35\:979& 363\:083&4\:717\:687
\end{array}$$
The set of finite spaces with $n$ elements will be denoted by $\SF_n$, and we put $\displaystyle \SF=\bigsqcup_{n\geq 0} \SF_n$. 
The vector space with basis given by the set of all finite spaces is written $\CF$ and its (finite dimensional) degree $n$ component, the subspace generated by finite spaces with $n$ elements, $\CF_n$. We will be from now on interested in the fine structure of $\CF$ in relation to classical topological properties and constructions. 

\subsection{Homotopy types}
The present section and the following survey the links between finite spaces and topological notions such as homotopy types. We refer to Stong's seminal paper \cite{Stong} and to Barmak's thesis \cite{BarmakThesis} on which this account is based for further details and references.

For a finite topological set, the three notions of connectedness, path-connectedness and order-connectedness agree (the later being understood as connectedness of the graph of the associated preorder).

For $f,g$ continuous maps between finite topological sets $X$ and $Y$, we set
$$f\leq g\Leftrightarrow \forall x\in X, \ f(x)\leq g(x) \mbox{\rm \ for the order of } Y.$$ This preorder on the (finite) mapping space $Y^X$ is the one associated to the compact-open topology. It follows immediately, among others, that two comparable maps are homotopic and that a space with a maximal or minimal element is contractible (since the constant map to this point will be homotopic to any other map -- in particular the identity map).

For the same reason, given a finite topological set $X$, there exists a homotopy equivalent finite topological set $X_0$ which is $T_0$ (the quotient space $X/\sim_\T$ considered in the previous section, for example).
Therefore, since \cite{Alexandroff},
the study of homotopy types of finite spaces is in general restricted to $T_0$ spaces. 
Characterizing homotopies (inside the category of finite topological sets) is also a simple task: 
two maps $f$ and $g$ are homotopic if and only if there exists a sequence:
$$f=f_0\leq f_1\geq f_2 \leq .... \geq f_n=g.$$

In the framework of finite topological sets, a reduction method refers to a combinatorial method allowing the removal of points
without changing given topological properties (such as the homotopy type).
Stong's reduction method allows a simple and effective construction of representatives of finite homotopy types \cite{Stong}. Stong first defines the
notions of linear and colinear points (also called up beat points and down beat points in a later terminology): a point $x\in X$ is linear
if $\exists y\in X, y>x$ and $\forall z>x, z\geq y$.
Similarly, $x\in X$ is colinear if $\exists y\in X, y<x$ and $\forall z<x, z\leq y$.
It follows from the combinatorial characterization of homotopies 
 that, if $x$ is a linear or colinear point in $X$, then $X$ is homotopy equivalent to $X-\{x\}$.

Together with the fact that any finite topological set is homotopy equivalent to a $T_0$ space, the characterization of homotopy types follows.
A space is called a core (or minimal finite space) if it has no linear or colinear points. By reduction to a $T_0$ space and recursive elimination of linear and colinear points, any finite topological set $X$ is homotopy equivalent to a core $X_c$ that can be shown to be unique up to homeomorphism \cite[Thm. 4]{Stong}.

\subsection{Simplicial realizations}

Another important tool to investigate topologically finite spaces is through their connection with simplicial complexes. 
We survey briefly the results of McCord, following \cite{McCord2,BarmakThesis}.

Recall that a weak homotopy equivalence between two topological spaces $X$ and $Y$
is a continuous map $f:X\rightarrow Y$ such that for all $x\in X$ and all $i\geq 0$, the induced map $f_\ast:\pi_i(X,x)\longmapsto \pi_i(Y,f(x))$ is an isomorphism (of groups for $i>0$).
The finiteness requirement enforces specific properties of finite spaces: for example, contrary to what happens for CW-complexes (Whitehead's theorem), there are weakly homotopy
equivalent finite spaces with different homotopy types.

The key to McCord's theory is the definition of functors between the categories of finite topological sets and simplicial complexes (essentially the categorical nerve and the topological realization). Concretely, to a finite topological set $X$ is associated the simplicial complex ${\mathcal K}(X)$ of non empty chains of $X/\sim_\T$ (that is, sequences $x_1<...<x_n$ in $X/\sim_\T$). Conversely, to the simplicial complex ${\mathcal K}(X)$ is associated its topological realization $|{\mathcal K}(X)|$: the points $x$ of $|{\mathcal K}(X)|$ are the linear combinations $x=t_1x_1+...+t_nx_n,\ \sum\limits_{i=1}^nt_i=1, \ t_i>0$. Setting $Sup(x):=x_1$, McCord's fundamental theorem states that:
$$Sup:|{\mathcal K}(X)|\longmapsto X/\sim_\T$$ 
is a weak homotopy equivalence. In particular, $|{\mathcal K}(X)|$ is weakly homotopy equivalent to $X$. Notice also that ${\mathcal K}(X)$ and ${\mathcal K}(X^\ast)$, resp. $|{\mathcal K}(X)|$ and $|{\mathcal K}(X^\ast)|$ are canonically isomorphic: a finite space is always weakly homotopy equivalent to its dual.

\section{Sums and joins}

We investigate from now on operations on finite spaces. Besides their intrinsic interest and their connections to various classical topological constructions, they are meaningful for the problem of enumerating finite spaces (see e.g. \cite{Stanley,Sharp,Erne}). They will also later underly the construction of  $B_\infty$-algebra structures.\\

{\bf Notation.} Let $O\subseteq \N$ and let $n \in \N$. The set $O(+n)$ is the set $\{k+n\mid k\in O\}$.

\begin{defi}
Let $\T\in \TT_n$ and $\T'\in \TT_{n'}$ be standard representatives of $\overline\T\in\SF_n$ and $\overline{\T'}\in\SF_{n'}$.
\begin{enumerate}
\item The topology $\T.\T'$ is the topology on $[n+n']$ for which open sets are the sets $O \sqcup O'(+n)$, with $O \in \T$ and $O'\in \T'$. The finite space $\overline\T.\overline{\T'}$ is  $\overline{\T.\T'}.$
\item The topology $\T \succ \T'$ is the topology on $[n+n']$ for which open sets are the sets $O \sqcup [n'](+n)$, with $O \in \T$,
and $O'(+n)$, with $O'\in \T'$. The finite space $\overline\T\succ\overline{\T'}$ is $\overline{\T \succ \T'}.$
\end{enumerate}\end{defi}

We omit the proof that the products  $\overline\T.\overline{\T'}$ and $\overline\T\succ\overline{\T'}$ are well-defined and do not depend on the choice of a standard representative.

The first product is the sum (disjoint union) of topological spaces. 

The second one deserves to be called the join. 
Recall indeed that the join $A\ast B$ of two topological spaces $A$ and $B$ is the quotient of $[0,1]\times A\times B$ by the relations $(0,a,b)\sim (0,a,b')$ and $(1,a,b)\sim (1,a',b)$. For example, the join of the $n$ and $m$ dimensional spheres is the $n+m+1$-dimensional sphere. When it is defined that way, the join is not an internal operation on finite spaces. However, recall that
the join of two simplicial complexes $K$ and $L$ is the simplicial complex $K\ast L:=K\coprod L\coprod \{\sigma\cup\beta,\ \sigma\in K,\ \beta\in L\}$ and that the join operation commutes with topological realizations in the sense that (up to canonical isomorphisms) $|K\ast L|=|K|\ast |L|$.
It follows therefore from McCord's theory that, up to a weak homotopy equivalence, the product $\succ$ is nothing but (a finite spaces version of) the topological join.\\

We extend linearly the two products defined earlier to $\F$. 
Note that these products define linear maps from $\F_m \times \F_n$ to $\F_{m+n}$ for all $m,n \geq 0$. \\

By a slight abuse of notation, we will allow ourselves to denote finite spaces using the notation ($X$, $Y$,...) that was reserved till now for finite topological sets. Similarly, we will not discuss systematically the fact that some constructions on finite spaces have to be done by choosing finite topological sets representatives of the corresponding finite spaces, performing the construction on these representatives, and then moving back to the corresponding homeomorphism classes.  \\

{\bf Examples.} 
$$\tdtroisun{$1$}{$3$}{$2$}.\tddeux{$2$}{$1$}=\tdtroisun{$1$}{$3$}{$2$}\tddeux{$5$}{$4$},\:
\tdtroisun{$1$}{$3$}{$2$}\succ \tddeux{$2$}{$1$}=\pdcinq{$1$}{$3$}{$2$}{$5$}{$4$}.$$

The join of two circles (see in the previous section the minimal finite space representation of a circle) is a 3-sphere:
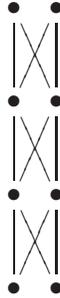
\begin{displaymath}
\xymatrix@C=4pt{  \bullet \ar@{-}[d] \ar@{-}[dr]  & \bullet \ar@{-}[d] \ar@{-}[dl] \\ 
		  \bullet \ar@{-}[d] \ar@{-}[dr]  & \bullet \ar@{-}[d] \ar@{-}[dl] \\ 
		  \bullet \ar@{-}[d] \ar@{-}[dr]  & \bullet \ar@{-}[d] \ar@{-}[dl] \\ 
		\bullet &  \bullet } 
\end{displaymath}
\begin{prop} 
These two products are associative, with $\emptyset=1$ as a common unit. The first product is also commutative. They are compatible with the duality involution:
$$X^\ast .Y^\ast =(X.Y)^\ast,\ Y^\ast \succ X^\ast =(X\succ Y)^\ast.$$
\end{prop}

The proof is left to the reader.

\begin{defi}
Let $X \in \SF$, different from $1$. Notice that $X$ is connected (or $.$-indecomposable) if and only if it cannot be written in the form $X=X'.X''$, with $X',X''\neq 1$.
\begin{enumerate}
\item We shall say that $X$ is join-indecomposable if it cannot be written in the form $X=X'\succ X''$, with $X',X''\neq 1$.
\item We shall say that $X$ is irreducible if it is both join-indecomposable and connected.
\end{enumerate} \end{defi}

{\bf Examples.} Here are the irreducible spaces of cardinality $\leq 4$:
$$\tun; \tdun{$2$}; \tdun{$3$}; \tdun{$4$}, \pquatresix.$$

The triple $(\F,.,\succ)$ is a $Com-As$ algebra, that is an algebra with a first commutative and associative product and a second, associative, product sharing the same unit. This is a particular example of a 2-associative algebra \cite{Loday}, that is to say an algebra with two associative products
sharing the same unit.

From now on, unless otherwise stated, space means finite space.

\begin{theo}\begin{enumerate}
\item The commutative algebra $(\F,.)$ is freely generated by the set of connected spaces.
\item The associative algebra $(\F,\succ)$ is freely generated by the set of join-indecomposable spaces.
\item The $Com-As$ algebra $(\F,.,\succ)$ is freely generated by the set of irreducible spaces.
\end{enumerate}\end{theo}

\begin{proof} Notice first the important property that the join-product of two non empty spaces is a connected space.

1. Any space can be written uniquely as a disjoint union of connected spaces.

2. Let $X$ be an arbitrary space, and let us choose one of its representatives (still written $X$).
Notice first that $X=Y\succ Z$ if and only if $Y<_\T Z$ (in the sense that, for arbitrary $y\in Y,\ z\in Z$, $y<_T z$).
That is, $$X=Y\succ Z\Leftrightarrow X=Y\coprod Z \textrm{\ \ and\ \ } Y<_\T Z.$$

Let us assume that $X=X_1\succ X_2 \succ ...\succ X_n=Y_1\succ Y_2\succ ...\succ Y_m$ with the $X_i$ and the $Y_j$ join-indecomposable. Then, $X_1\cap Y_1$ is not empty (this would imply for example that $Y_1\subset X_2\succ ...\succ X_n>_\T X_1$, and similarly $X_1 >_\T Y_1$, which leads immediately to a contradiction). Moreover, $X_1\cap Y_1<_T X_1\cap (Y_2\succ ...\succ Y_m)$. A contradiction follows if $X_1\cap Y_1\not= X_1,Y_1$ since we would then have
$$X_1=(X_1\cap Y_1)\succ ( X_1\cap (Y_2\succ ...\succ Y_m)) .$$
We get $X_1=Y_1$ and $X_2 \succ ...\succ X_n= Y_2\succ ...\succ Y_m$, and the statement follows by induction.

3. Let us describe briefly the free $Com-As$ algebra $CA(S)$ over a set $S$ of generators (we write $.$ and $\succ$ for the two products). A basis $B=S\coprod B_C\coprod B_A$ of $CA(S)$ with  $B_C=\coprod\limits_{n\geq 2}B_{C,n}$,  $B_A=\coprod\limits_{n\geq 2}B_{A,n}$ can be constructed recursively as follows (in the following $B_{A,1}=B_{C,1}:=S$ and the $.$ product is commutative so that $a.b=b.a$):
\begin{itemize}
 \item $B_{C,n}:=\coprod\limits_{n_1+...+n_k=n}\{a_1. \ ...\ .a_k,\ a_i\in B_{A,n_i}\}$
 \item $B_{A,n}:=\coprod\limits_{n_1+...+n_k=n}\{a_1\succ \ ...\ \succ a_k,\ a_i\in B_{C,n_i}\}.$
\end{itemize}

Now, let $X$ be a space, then one and only one of the three following cases holds
\begin{enumerate}
 \item $X$ is irreducible.
 \item $X$ is $.$-indecomposable and join-decomposable, and then it decomposes uniquely into a product $X=X_1\succ ...\succ X_k$ of join-indecomposable spaces.
 \item $X$ is $.$-decomposable and join-indecomposable, and then it decomposes uniquely into a sum $X=X_1\cup ...\cup X_k$ of connected spaces.
\end{enumerate}
It follows by induction that the set of spaces identifies with the basis of the free $Com-As$ algebra over irreducible spaces: writing $S$ for the latter set, the first case in the previous list corresponds to the case $X\in S$; the second to $X\in B_{A}$ with the $X_i$ in $B_C$ or $S$; the third to $X \in B_{C}$ with the $X_i$ in $B_A$ or $S$.
\end{proof}

As an application, it is possible to obtain the numbers $p_n$, $q_n$ and $r_n$ of, respectively, connected, join-indecomposable and irreducible spaces
of cardinality $n$, by manipulating formal series. This gives:

$$\begin{array}{c|c|c|c|c|c|c|c|c|c|c}
n&1&2&3&4&5&6&7&8&9&10\\
\hline p_n&1&2&6&21&94&512&3\:485&29\:515&314\:474&4\:255\:727\\
\hline q_n&1&2&4&14&62&373&2\:722&24\:591&275\:056&3\:860\:200\\
\hline r_n&1&1&1&2&17&167&1\:672&18\:127&226\:447&3\:398\:240
\end{array}$$

The sequences $(p_n)$, $(q_n)$ and  $(r_n)$ are A001928,  A046911 and A046909 of \cite{Sloane}, see also \cite{Wright}.

\section{Schur-Weyl categories}

We will show, in forthcoming sections, that $\F$ carries various bialgebraic structures. Rigidity theorems in the sense of \cite{Livernet} apply, making $\F$ a cofree Hopf algebra and the $B_\infty$-enveloping algebra of a $B_\infty$-algebra (see Theorem~\ref{th9}). 

There are various ways to give an algebraic and combinatorial characterization of cofree Hopf algebras, following ideas that are scattered in the litterature and seem to originate in the Bott-Samelson theorem, according to which $H_\ast(\Omega\Sigma X;K)=T^c(H_\ast (X;K))$, where $\Sigma$ is the suspension functor acting on topological spaces and $\Omega$ the loop space functor, and in the work of Baues on the bar/cobar construction \cite{baues1981double}, \cite[p. 48]{getzler1994operads}. The paper \cite{Loday} addresses the problem explicitly, but other approaches follow from \cite{Berstein,fresse1998algebre,patras1999leray,Livernet}, and no unified treatment seems to have been given up to date. 
We take the opportunity of the present article and the existence of such structures on finite spaces to present such a short and self-contained treatment. In the process, we extend the results of Livernet \cite{Livernet} on cocommutative cogroups in the category of associative algebras
and infinitesimal bialgebras and our study in \cite{foissy2012natural} of natural operations on shuffle bialgebras. 

In the present section, we focus on free associative algebras and cofree coassociative coalgebras. Recall that we work with graded connected structures: in the following, $Vect$ stands for the category of connected graded vector spaces $V=\bigoplus\limits_{n\in\N}  V_n$, where $V_0={0}$ is the null vector space;
if $V$ and $W$ are two objects of $Vect$, a morphism $f:V\longrightarrow W$ in $Vect$ is a linear map, homogeneous of degree $0$, that is to say
$f(V_n) \subseteq W_n$ for a	ll $n\geq 0$. 
 We write $|v|=n$ if $v$ is a (non-zero) homogeneous element of degree $n$ in $V$. The category of connected graded vector spaces augmented with the ground field $K$ in degree $0$ will be written $Vect^+$ (i.e., for $V=\bigoplus\limits_{n\in\N}  V_n\in Vect^+, \ V_0=K$); if $V$ and $W$ are two objects in 
 $Vect^+$, a morphism $f:V\longrightarrow W$ is a linear map, homogeneous of degree $0$, such that $f_{\mid K}=Id_K$.
We write $\epsilon$ for the canonical projection to $V_0=K$, 
 $T(V) =\bigoplus\limits_{n\in\N}T^n(V):=\bigoplus\limits_{n\in\N}V^{\otimes n}\in Vect^+$, $\overline{T}(V):= \bigoplus\limits_{n\in\N^\ast}V^{\otimes n}\in Vect$ and call $T(V)$ the tensor space over $V$ (resp. $T^n(V)$ the space of tensors of length $n$ over $V$).
We use the shortcut notation $v_1 ... v_n$ for $v_1\otimes ...\otimes v_n\in V^{\otimes n}$ and will call sometimes $v_1 ... v_n$ a word (of  length $n$) over $V$.
Notice that the grading of $T(V)$ by the length differs from the grading $T(V) =\bigoplus\limits_{n\in\N}T_n(V)$ canonically induced by the grading of $V$ ($T_n(V)$ being generated by words $v_1...v_k$ where the $v_i$s are homogeneous with $|v_1|+...+|v_k|=n$). 

We are now in the position to recall the definition of the algebra of graded permutations. This algebra plays the role, for the categories of bialgebras that we are going to study, that the descent algebra plays for usual bialgebras (graded connected commutative or cocommutative bialgebras), see \cite{patras1994algebre}. This point of view was developed in \cite{foissy2012natural} for shuffle bialgebras and is extended here to other, naturally equivalent, categories. We write $\sym_k$ for the symmetric group on $k$ elements.

\begin{defi}\label{gradperm}
Let us fix $k \in \mathbb{N}$. Let $\sigma \in \sym_k$ and $d:[k]\longrightarrow \mathbb{N}_{>0}$.
We define a linear endomorphism of $T(V)$ by:
$$\Phi_{(\sigma,d)}:\left\{\begin{array}{rcl}
v_1\ldots v_l&\longrightarrow&v_{\sigma(1)}\ldots v_{\sigma(l)} \mbox{ if }k=l \mbox{ and }|v_{\sigma(i)}|=d(i) \mbox{ for all }i,\\
&\longrightarrow&0\mbox{ if not.}
\end{array}\right.$$
\end{defi}

The composition of graded permutations is given as follows: for all $(\sigma,d) \in \sym_k \times Hom([k],\mathbb{N}_{>0})$ and $(\tau,e) \in \sym_l \times Hom([l],\mathbb{N}_{>0})$,
$$\Phi(\sigma,d)\circ \Phi(\tau,e)=\left\{\begin{array}{l}
\Phi(\tau \circ \sigma,d)\mbox{ if }k=l \mbox{ and }d=e\circ \sigma,\\
0\mbox{ if not.}
\end{array}\right.$$

{\bf Notations.}
We put $\displaystyle \S=\coprod_{k\geq 0} \sym_k \times Hom([k],\mathbb{N}_{>0})$, and $\VS=Vect(\S)$, the vector space generated by $\S$.
The composition of $\S$, linearly extended, makes $\VS$ an algebra. 

Following the method used in \cite{NPT} to study nonlinear Schur-Weyl duality, we want to characterize natural
transformations of the functor
\begin{equation}
T(V):=\bigoplus\limits_{n\in\N}T_n(V):=\bigoplus\limits_{n\in\N}V^{\otimes n}
\end{equation}
viewed as a functor from $Vect$ to $Vect^+$. Concretely, we look for $V\in Vect$-indexed families of graded linear maps $\mu_V$ from $T_n(V)$
to $ T_n(V)$ where $n$ is an arbitrary integer such that, for any map $f$ of graded vector spaces from $V$
to $W$,
\begin{equation}
\label{natu}
T_n(f)\circ \mu_V=\mu_W\circ T_n(f).
\end{equation}
Let us say that such a family $\mu_V$ satisfies graded Schur-Weyl duality in degree $n$ (by extension of the classical case, where the ground field is of characteristic $0$ and the problem is restricted to invertible endomorphisms $f$ of a given non graded vector space $V$).

\begin{prop}\label{schurweyl}
Let $SW$ be the vector space spanned by families of linear maps that satisfy
the graded Schur-Weyl duality in degree $n$, where $n$ runs over $\N^\ast$. The $SW$ is canonically isomorphic to $\VS$.
Moroever, the composition of natural transformations makes $SW$ an algebra, and the canonical isomorphism from $SW$ to $\VS$ is an algebra isomorphism.
\end{prop}

\begin{proof}  The action of $\VS$ on the tensor spaces $T(V)$s is natural (it commutes with an arbitrary $T(f)$): $\VS$ is, as an algebra, canonically
embedded in $SW$.

Let $\mu$ be a family of linear maps satisfying the graded Schur-Weyl duality in degree $n$.
For any finite sequence $d=(d_1,\ldots,d_n)$ of elements of $\N_{>0}$, let us put $X_d=Vect(x_1,\ldots,x_n)$, where $|x_i|:=d_i$ for all $i$.
For an arbitrary family $a_1,\ldots,a_n$ of elements of a graded vector space $V$ with $|a_i|=d_i$ for all $i$, the map $f(x_i):=a_i$
extends uniquely to a linear map from $X_d$ to $V$. Then:
$$\mu_V(a_1\ldots a_n)=\mu_V\circ T(f)(x_1\ldots x_n)=T(f)(\mu_{X_d}(x_1\ldots x_n)),$$
so  the knowledge of the elements $x_d:=\mu_{X_d}(x_1... x_n)$ for any $d$ determines entirely $\mu$.

Let us fix $d$. As $x_d \in T_n(X_d)$, we can write:
$$x_d=\sum_{\sigma:[n]\longrightarrow [n]} a_{\sigma,d} x_{\sigma(1)}\ldots x_{\sigma(n)}.$$
Let $i\in [n]$. We define $f_i:X_d\longrightarrow X_d$ by $f_i(x_j)=x_j$ if $i\neq j$ and $0$ if $i=j$. Then:
$$0=\mu_{X_d} \circ T_n(f_i)(x_1\ldots x_n)=T_n(f_i)(x_d)=\sum_{\substack{\sigma:[n]\longrightarrow [n]\\ i\notin \sigma([n])}} 
a_{\sigma,d} x_{\sigma(1)}\ldots x_{\sigma(n)}.$$
Hence, if $\sigma([n]) \subsetneq [n]$, $a_{\sigma,d}=0$, so:
$$x_d=\sum_{\sigma \in \S_n} a_{\sigma,d} x_{\sigma(1)}\ldots x_{\sigma(n)}
=\sum_{(\sigma,e)\in \mathfrak{S}} a_{\sigma,e}\Phi_{(\sigma,e\circ \sigma)}(x_1\ldots x_n).$$
which implies that:
$$\mu=\sum_{(\sigma,e)\in \mathfrak{S}} a_{\sigma,e\circ \sigma^{-1}}\Phi_{(\sigma,e)},$$
and finally $SW=\Phi(\VS)$. \end{proof}


%



Building on these results, we define Schur-Weyl categories. 

\begin{defi}
A Schur-Weyl category is a category $\mathcal C$ with a forgetful functor $F$ to $Vect^+$ (i.e. whose objects are naturally equipped with a structure of graded vector spaces), with a functor $P$ to $Vect$, and with natural isomorphisms: 
$$\forall C\in {\mathcal C}, I(C):T\circ P(C)\cong F(C)$$
(i.e. the objects of $\mathcal C$ are naturally isomorphic to tensor spaces).
In particular, due to Proposition~\ref{schurweyl}, the objects of $\mathcal C$ are naturally equipped with an action of $\S$.
\end{defi}

Recall now the definition of various algebraic structures on the tensor spaces $T(V)$. We point out that the Proposition~\ref{schurweyl} shows that these 
structures (which can be described as the composite of natural endomorphisms of the functors $T_n$ with the natural isomorphisms $T_{m+n}\cong T_m\otimes T_n$) are naturally defined. More generally, the Proposition shows that the definition of \it natural \rm (graded) algebraic structures on the tensor spaces is constrained by the graded Schur-Weyl duality phenomenon: in concrete terms, one has to use permutations to define such structures. At last, notice that these structures will lift automatically to Schur-Weyl categories.

\begin{itemize}
\item The tensor algebra over $V$ is the tensor space over $V$ equipped with the concatenation product:
$$v_1...v_n \cdot w_1...w_m:=v_1...v_nw_1...w_m,$$
the tensor algebra is the free associative algebra over $V$.
\item The tensor coalgebra over $V$ is the tensor space over $V$ equipped with the deconcatenation coproduct $\Delta$, so that $$\Delta(v_1 ... v_n):=\sum\limits_{i=0}^nv_1...v_i\otimes v_{i+1}...v_n,$$
it is the cofree connected coassociative coalgebra over $V$ (the general structure of cofree coalgebras is more subtle, see \cite{hazewinkel2010algebras}). 
\item The shuffle algebra over $V$  is the tensor space over $V$ equipped with the shuffle product:
$$v_1...v_n\shuffle w_1...w_m:=\sum\limits_{\sigma}x_{\sigma^{-1}(1)}...x_{\sigma^{-1}(n+m)},$$
where $x_1...x_{n+m}:=v_1...v_nw_1...w_m$ and $\sigma$ runs over the $(n,m)$-shuffles in $\sym_{n+m}$, that is over the permutations such that:
$\sigma(1)<...<\sigma(n),\sigma(n+1)<...<\sigma(n+m).$
It is the free shuffle algebra over $V$ (see below for a definition).
\item The unshuffle coalgebra over $V$ is the tensor space over $V$ equipped with the unshuffle coproduct:
$$\delta(v_1 ... v_n):=\sum\limits_{I,J}v_{i_1}...v_{i_k}\otimes v_{j_1}...v_{j_{n-k}},$$
where $I=\{i_1,...,i_k\},J=\{j_1,...,j_{n-k}\}$ run over all partitions of $[n]$ into two disjoint (and possibly empty) subsets. It is the cofree unshuffle coalgebra over $V$ (see below for a definition).
\end{itemize}

Recall that, given an arbitrary  coproduct map $\Delta$ from $X$ to $X\otimes X$ in $Vect^+$, the associated vector space of primitive elements is defined by $Prim(X):=\{x\in X,\Delta(x)=x\otimes 1+1\otimes x\}$. For the deconcatenation coproduct on $T(V)$, we have $Prim(T(V))=V$, whereas for the unshuffle coproduct, over a field of characteristic $0$, $Prim(T(V))=Lie(V)$, the free Lie algebra over $V$, see e.g. \cite{reutenauer1993free}.

\section{Schur-Weyl categories of bialgebras}

The objects of a Schur-Weyl category $\mathcal C$ are naturally equipped with these four algebra and coalgebra structures. 
Let us go now one step further and investigate tensor spaces from the point of view of bialgebras.
The four algebra/coalgebra maps give rise to three interesting bialgebra structures.

\subsection{Shuffle bialgebras}

Recall first from \cite{schutzenberger1958propriete}
that the shuffle product $\shuffle$ is characterized abstractly in $Vect$ by the identity involving the left and right half-shuffles $\prec, \succ$ (with $\shuffle = \prec + \succ$):
\begin{equation}\label{shuf}
 x\prec y = y\succ x,\ (x\prec y)\prec z = x\prec (y\prec z + y\succ z).
\end{equation}
This definition is extended to $Vect^+$ by requiring $x\prec 1=x,\ 1\prec x=0$, see e.g. \cite{foissy2012natural} for details.

\begin{defi}
\begin{enumerate}
\item A shuffle bialgebra is a commutative Hopf algebra whose product, written $\shuffle$ is a shuffle product (that is, can be written $\shuffle = \prec +\succ$ in such a way that $\prec $ and $\succ$ satisfy the identities (\ref{shuf})) and, for $x,y\in Ker\ \epsilon$,  the extra axiom:
\begin{align*}
\Delta(x\prec y)&=x\prec y\otimes 1+1\otimes x\prec y+x\otimes y+x\prec y'\otimes y''\\
&+x'\prec y\otimes x''+x'\otimes x''\shuffle y+x'\prec y'\otimes x''\shuffle y'',
\end{align*}
where we use Sweedler's notation $\Delta(x)=x_1\otimes x_2=x\otimes 1+1\otimes x+x'\otimes x''$.
We shorten this axiom as:
$$\Delta(x\prec y)=x_1\prec y_1\otimes x_2\shuffle y_2.$$
\item If $A$ and $B$ are two shuffle bialgebras, a morphism of shuffle bialgebras $f:A\longrightarrow B$ is a Hopf algebra morphism from $A$ to $B$,
homogeneous of degree $0$, such that for all  $x,y \in Ker\epsilon$:
\begin{align*}
f(x\prec y)&=f(x)\prec f(y),& f(x\succ y)&=f(x)\succ f(y).
\end{align*}
\end{enumerate}

\end{defi}

The tensor space $T(V)$ is equipped with the structure of a shuffle bialgebra by the deconcatenation coproduct $\Delta$ and the left and right half-shuffle maps  $\prec$, $\succ$  (they add up to $\shuffle$)
defined recursively by:
$$x_1\prec y_1:=x_1y_1,\ \ x_1...x_n\prec y_1...y_m:=x_1(x_2...x_n\shuffle y_1...y_m),$$
$$x_1\succ y_1:=y_1x_1,\ \ x_1...x_n\succ y_1...y_m:=y_1(x_1...x_n\shuffle y_2...y_m).$$

\begin{theo}\label{SBrigidity}
A shuffle bialgebra $B$ is  isomorphic as a shuffle bialgebra to $T(Prim(B))$, where $Prim$ stands for the functor of primitive elements from the category $\mathcal{SB}$  of shuffle bialgebras to $Vect$.
Furthermore, the category $\mathcal{SB}$ is a Schur-Weyl category.
 \end{theo}

The first part of the Theorem is a rigidity theorem (recall that $T(V)$ is free as a shuffle algebra and cofree as a coassociative coalgebra in $Vect^+$) and a consequence of a more general result \cite[Prop. 15]{chapoton2002theoreme}. Another, direct, proof of this result, based on compositions in the free magmatic algebra, was obtained in \cite{Burgunder}[Appendix A]. Unfortunately, in spite of many important insights on the behaviour of shuffle bialgebras, the proof is not entirely conclusive (the composition of formal power series argument at the end of the Appendix does not apply). As the author pointed out recently to one of us, an alternative strategy of proof can however  be developed sticking inside her magmatic approach.

The precise form of the Theorem, as stated here (including a construction of a natural isomorphism from  $T(Prim(B))$ to $B$, as required in a Schur--Weyl category) is 
obtained in \cite{foissy2012natural}[Thm. 6.7] (the proof contains the effective construction of the natural isomorphism). 

\subsection{Unshuffle bialgebras}
Dually, one can split the unshuffle coproduct $ \delta = \delta_\prec + \delta_\succ$ on $ \overline{T}(V)$:
 for $xX=xx_1...x_n,\ x,...,x_n\in V$,
$$  \delta_\prec(x):=x\otimes 1,\ \ \delta_\succ (x):=1\otimes x;$$
$$\delta_\prec (xX):=xX_1\otimes X_2,\ \delta_\succ (xX):=X_1\otimes xX_2,$$
where we use Sweedler's notation $\delta(X)=X_1\otimes X_2$.
Notice that $V=Prim_\prec(T(V))$, where
$$Prim_\prec(T(V)):=\{b\in \overline{T}(V),\delta_\prec(b)=b\otimes 1\}.$$

The left and right  half-unshuffles $\delta_\prec , \delta_\succ$
satisfy the identities:
\begin{equation}\label{unshuff}
\delta_\prec = \tau\circ \delta_\succ,\ (\delta_\prec\otimes Id)\circ \delta_\prec=(Id\otimes \delta)\circ \delta_\prec ,
\end{equation}
where $\tau$ stands for the switch map $\tau(x\otimes y)=y\otimes x$, and, on $Ker(\epsilon)$,
\begin{equation}\label{coununsh}
(\epsilon\otimes Id)\circ \delta_\prec (x)=0,\ (Id\otimes \epsilon)\circ \delta_\prec (x)=x,
\end{equation}
 where we recall that $\epsilon$ stands for the augmentation (the canonical projection to the ground field) in $Vect^+$.

\begin{defi} \begin{enumerate}
\item Using the shortcut $\delta_\prec(x)=x_1^\prec \otimes x_2^\prec$ (and similarly for $\delta_\succ$), an unshuffle bialgebra is a bialgebra equipped with a coassociative cocommutative coproduct
$\delta = \delta_\prec + \delta_\succ$ satisfying the above identities and an associative product $\cdot$ such that furthermore, for $x,y\in Ker\ \epsilon$:
\begin{equation}\label{coshuff}
\delta_\prec(x\cdot y)= x_1^\prec\cdot y_1\otimes x_2^\prec \cdot y_2.
\end{equation}
\item If $A$ and $B$ are two unshuffle algebras, a morphism of unshuffle algebras $f:A\longrightarrow B$ is a bialgebra morphism from $A$ to $B$,
homogeneous of degree $0$ such that:
\begin{align*}
\delta_\prec \circ f&=(f\otimes f)\circ \delta_\prec,&
\delta_\succ \circ f&=(f\otimes f)\circ \delta_\succ.
\end{align*}\end{enumerate}\end{defi}

The tensor space equipped with the concatenation product and the two half-unshuffles  $\delta_\prec , \delta_\succ$ described previously is an unshuffle bialgebra. 

The two notions of shuffle bialgebras and unshuffle bialgebras are strictly dual (in the graded sense - the graded dual of a vector space $V=\bigoplus\limits_{n\in \N}V_n$  in $Vect^+$ being the direct sum of the duals $V^\ast=\bigoplus\limits_{n\in \N}V_n^\ast$).

The rigidity theorem for unshuffle bialgebras follows by duality (this was first observed in \cite{Burgunder}[Appendix B]): they are isomorphic to free associative algebras and cofree unshuffle coalgebras.
The natural isomorphisms defining Schur-Weyl duality can be obtained by dualizing the constructions  in \cite{foissy2012natural}, in particular Corollary 3.3 on which the later proof of the structure theorem for shuffle bialgebras relies in that article.
Let us  sketch the proof of the analogue of this Corollary - the rest of the construction of the natural isomorphisms is left to the reader, we refer to  \cite{foissy2012natural} for details.

Let  $f,g$ be two endomorphisms of $T(V)$ in $Vect^+$. We set:
$$f\prec g(x):=f( x_1^\prec)g( x_2^\prec), \ f\succ g(x):=f( x_1^\succ)g( x_2^\succ),$$
$$f\shuffle g(x):=f\prec g(x)+ \ f\succ g(x).$$
The two half-products $\prec$ and $\succ$ define the structure of a noncommutative shuffle algebra (or dendrimorphic) algebra on $End(T(V))$, that is they satisfy the identities $$(f\prec g)\prec k=f\prec (g\shuffle k),\ (f\shuffle g)\succ k=f\succ (g\succ k),$$
$$(f\succ g)\prec k=f\succ (g\prec k).$$
The first of these identities follows directly, for example, from $ (\delta_\prec\otimes Id)\circ \delta_\prec=(Id\otimes \delta)\circ \delta_\prec$, and similarly for the others.
Let us write now $\pi$ for the projection from $T(V)$ to $V$ orthogonally to the other components.

\begin{lemma}
We have, in $End(T(V))$,
\begin{equation}\label{fundeqq}
Id=\epsilon +\sum\limits_{n\in\N^\ast}\pi\prec (\pi\prec (...(\pi\prec \pi)...)).
\end{equation}
\end{lemma}

Indeed, for $X:=v_1...v_n\in V^{\otimes n}$, 
$\delta_\prec (X)$ equals $v_1\otimes v_2...v_n$ plus a remainder term $R$ such that $\pi\otimes Id(R)=0$.
We get: $Id=\epsilon+\pi\prec Id$, from which the Lemma follows by a perturbative expansion. Let us mention that the latter equation can investigated systematically, see for example \cite{ebrahimi2009dendriform}.

When written in $End(B)$, for $B$ an arbitrary unshuffle bialgebra, the equation (\ref{fundeqq}) defines (implicitely) $\pi$. The iterated products $\pi\prec (\pi\prec (...(\pi\prec \pi)...))$ are then the analogues, on $B$, of the projections from $T(V)$ to the summand $V^{\otimes n}$ orthogonally to the other components.

\begin{theo}\label{UBrigidity}
An unshuffle bialgebra $B$ is  isomorphic as an unshuffle bialgebra to $T(Prim_\prec(B))$.
Furthermore, the category $\mathcal{UB}$ of unshuffle bialgebras  is a Schur-Weyl category.
 \end{theo}

\subsection{Infinitesimal bialgebras}

The coproduct $\ast$ in the category $\mathcal As$ of (unital) associative algebras in $Vect^+$, or free product, is obtained as follows: let $H_1=K\oplus \overline H_1,H_2=K\oplus \overline H_2$ be two such algebras, then:
$$H_1\ast H_2:=K\oplus\bigoplus\limits_{n\in\N^\ast}(H_1\ast H_2)^{(n)}:=K\oplus \bigoplus\limits_{n\in\N^\ast}[(1,H^{\otimes n})\oplus (2,H^{\otimes n})],$$
where $(1,H^{\otimes n})$ (resp. $(2,H^{\otimes n})$) denotes alternating tensor products of $\overline H_1$ and $\overline H_2$ of length $n$ starting with $\overline H_1$
(resp. $\overline H_2$). 
For example, $(2,H^{\otimes 4})=\overline H_2\otimes \overline H_1\otimes \overline H_2\otimes \overline H_1$.
The product of two tensors $h_1\otimes ...\otimes h_n$ and $h_1'\otimes ...\otimes h_m'$ in $H_1\ast H_2$ is defined as the concatenation product $h_1\otimes ...\otimes h_n\otimes h_1'\otimes ...\otimes h_m'$ when $h_n$ and $h_1'$ belong respectively to $\overline H_1$ and $\overline H_2$ (or to $\overline H_2$ and $\overline H_1$), and otherwise as:
$h_1\otimes ...\otimes (h_n\cdot h_1')\otimes ...\otimes h_m'$.

When $H_1= T(V_1)$ and $H_2= T(V_2)$, one gets $H_1\ast H_2= T(V_1\oplus V_2)$. Moreover, by universal properties of free algebras, the linear map $\iota$ from $V$ to $T(V)\ast  T(V)$ defined by
\begin{equation}\label{iota}
\iota(v):=(1,v)+(2,v)
\end{equation}
induces an algebra map from $ T(V)$ to $ T(V)\ast  T(V)$ which is associative, unital ($\iota(x)=(1,x)+(2,x)+z$ with $z\in \bigoplus\limits_{n\geq 2}(H_1\ast H_2)^{(n)}$) and cocommutative. 
Equivalently, $ T(V)$ is a cocommutative cogroup in $\mathcal As$. 

\begin{defi}\begin{enumerate}
\item An infinitesimal bialgebra is, equivalently
\begin{itemize}
\item A cogroup in the category of associative unital algebras in $Vect^+$,
\item An associative unital algebra with product $\cdot$ and a coassociative counital coalgebra with coproduct $\Delta$  in $Vect^+$
such that furthermore, with the notation $\Delta(x)=x_1\otimes x_2$,
\begin{equation}\label{coprod2}
 {\Delta}(x\cdot y)=x\cdot {y_1}\otimes {y_2}+{x_1}\otimes{x_2}\cdot y-x\otimes y
\end{equation}
\end{itemize}
\item If $A$ and $B$ are two infinitesimal bialgebras, a morphism of infinitesimal bialgebras $f:A\longrightarrow B$ is a linear map from $A$ to $B$,
homogeneous of degree $0$, both an algebra and a coalgebra morphism.
\end{enumerate}
\end{defi}

The equivalence between these two definitions is not widely known: it is due to Livernet \cite{Livernet}; it is similar (in all respects) to the equivalence between cocommutative cogroups in the category of commutative algebras in $Vect^+$ and bicommutative bialgebras.
The equivalence follows from the observation that the structure map $\phi: H\longrightarrow H\ast H$ of such a cocommutative cogroup  is entirely determined by its restriction $\overline\Delta$ to its  image on the component $(1,H\otimes H)\cong \overline H\otimes \overline H$ of $H\ast H$. Namely,
\begin{equation}\label{phi}
 \phi(a)=\sum\limits_{n\geq 1}(1,\overline{\Delta}^{[n-1]}(a))+(2,\overline{\Delta}^{[n-1]}(a)),
\end{equation}
where $\overline{\Delta}^{[n-1]}$ stands for the iterated (coassociative) coproduct from $\overline H$ to $\overline H^{\otimes n}$. Using the notation $\overline{\Delta}(x)=\overline{x_1}\otimes \overline{x_2}$ (and more generally $\overline{\Delta}^{[n-1]}(x)=\overline{x_1}\otimes ...\otimes \overline{x_n}$), the coproduct $\overline\Delta$ satisfies the identity
\begin{equation}\label{coprod}
 \overline{\Delta}(x\cdot y)=x\otimes y+x\cdot \overline{y_1}\otimes \overline{y_2}+\overline{x_1}\otimes \overline{x_2}\cdot y
\end{equation}
so that, for $\Delta (x):=\overline{\Delta}(x)+x\otimes 1+1\otimes x$, with the notation ${\Delta}(x)={x_1}\otimes {x_2}$ we get the identity (\ref{coprod2}).
Note that in the case of $T(V)$, ${\Delta}$ is the deconcatenation coproduct.

Conversely, the identity (\ref{coprod}) is enough to ensure that (with the notation $\overline{\Delta}^{[0]}(x)=x=\overline x$)
$$\overline{\Delta}^{[k]}(x\cdot y)= \sum\limits_{i=1}^k\overline x_1\otimes ...\otimes \overline x_i\otimes \overline y_1\otimes ...\otimes \overline y_{k+1-i}$$
$$+\sum\limits_{i=1}^{k+1} \overline x_1\otimes ...\otimes \overline x_i\cdot \overline y_1\otimes ...\otimes \overline y_{k+2-i},$$
from which it follows that $ \phi $, as defined by the equation (\ref{phi}) defines a cogroup structure on $H$. 

\begin{theo}\label{IBrigidity}
An infinitesimal bialgebra $B$ is  isomorphic as an infinitesimal bialgebra to $T(Prim(B))$.
Furthermore, the category $\mathcal{IB}$ of infinitesimal bialgebras is a Schur-Weyl category.
 \end{theo}

The first rigidity statement is Berstein's structure theorem for cocommutative cogroups in categories of associative algebras \cite[Cor. 2.6]{Berstein} and Thm. 2.6 of \cite{Loday}. It implies that an infinitesimal bialgebra is free as an associative algebra and cofree as a coassociative coalgebra. The second statement follows from the proof of Theorem 2.6 in \cite{Loday}.

\subsection{Equivalence between Schur-Weyl categories of bialgebras}

We already noticed that any object of a Schur-Weyl category is naturally equipped with the structures of an associative algebra, of a shuffle algebra, of a coassociative coalgebra and of an unshuffle coalgebra. The same arguments show that it is naturally equipped with the structure of a shuffle bialgebra, of an unshuffle bialgebra and of an infinitesimal bialgebra.

The three structure theorems for shuffle, unshuffle and infinitesimal bialgebras imply the fundamental structure theorem:

\begin{theo}\label{theofund}
The categories of shuffle bialgebras, unshuffle bialgebras and infinitesimal bialgebras are isomorphic over $Vect^+$. They are all equipped with a natural action of the algebra of graded permutations $\VS$.
\end{theo}

By \it  isomorphic over $Vect^+$\rm , we mean that the three categories are equivalent, and that the natural equivalences can be realized as natural isomorphisms of graded vector spaces (concretely, an object of any of the three categories viewed as an element of $Vect^+$ can be equipped naturally with the other two  bialgebra  structures).

\begin{proof} Let us define for example the equivalence between  the category of infinitesimal bialgebras  and the category of shuffle bialgebras.
Let $A$ be an infinitesimal bialgebra. Denoting by $V$ the graded space of primitive elements, there exists a unique morphism
of infinitesimal bialgebras $f_A:A\longrightarrow T(V)$. As $T(V)$ is also a shuffle bialgebra, via the bijection $f_A$, $A$ becomes
an infinitesimal bialgebra in a unique way. This defines the image of $A$ by the equivalence. \end{proof}

\ \par

Let us show concretely how this process can be realized in practice on the example of infinitesimal bialgebras and unshuffle bialgebras -- we will explain later on how this example allows an improvement of the understanding of  one of Berstein's key notions: the one of the underlying algebra of a cocommutative cogroup in the category of associative algebras in $Vect^+$.

Let $H$ be such a cocommutative cogroup. The structure map $\phi : H\longrightarrow H\ast H$ gives rise
to two ``half-coproducts'' $\delta_\prec , \delta_\succ$ from $H$ to $H\otimes H$ defined as follows. 
Let $h_1\otimes ...\otimes h_n\in (H\ast H)^{(n)}$, we set: $$\pi_1(h_1\otimes ...\otimes h_n):=1_{h_1\otimes ...\otimes h_n\in (1,H^{\otimes n})}h_1\cdot h_3\cdot ...\cdot h_{n-1}\otimes h_2\cdot h_4\cdot ...\cdot h_n$$
$$\pi_2(h_1\otimes ...\otimes h_n):=1_{h_1\otimes ...\otimes h_n\in (2,H^{\otimes n})}h_2\cdot h_4\cdot ...\cdot h_n\otimes h_1\cdot h_3\cdot ...\cdot h_{n-1}$$
if $n$ is even and otherwise $$\pi_1(h_1\otimes ...\otimes h_n):=1_{h_1\otimes ...\otimes h_n\in (1,H^{\otimes n})}h_1\cdot h_3\cdot ...\cdot h_{n}\otimes h_2\cdot h_4\cdot ...\cdot h_{n-1},$$
$$\pi_2(h_1\otimes ...\otimes h_n):=1_{h_1\otimes ...\otimes h_n\in (2,H^{\otimes n})} h_2\cdot h_4\cdot ...\cdot h_{n-1}\otimes h_1\cdot h_3\cdot ...\cdot h_{n}.$$
Then, $$\delta_\prec(h):=\pi_1\circ\phi(h),\ \delta_\succ(h):=\pi_2\circ\phi(h).$$
Maps $\pi_i,\ i=1,2,3$ from $H\ast H\ast H$ to $H\otimes H\otimes H$ are defined similarly. That is, distinguishing notationally between the three copies of $H$ by writing $H\ast H\ast H=H_1\ast H_2\ast H_3$, $\pi_1$ acts non trivially on $h_1\otimes ...\otimes h_n\in H_1\ast H_2\ast H_3$ if and only if $h_1\in H_1$, and so on.

\begin{prop}\label{halfcop}
 The half-coproducts $\delta_\prec , \delta_\succ$ together with the associative product define (functorially) on $H$ the structure of an unshuffle bialgebra.
\end{prop}

The identity $\delta_\prec = \tau\circ \delta_\succ$ follows from the cocommutativity of $\phi$.
The identity
$(\delta_\prec\otimes Id)\circ \delta_\prec=(Id\otimes \delta)\circ \delta_\prec$ follows by observing that 
both maps act as 
$\pi_1\circ \phi^{[3]}$ on $H$, where $\phi^{[3]}$ is the iterated coproduct from $H$ to $H\ast H\ast H$.
The identity (\ref{coshuff}) follows from the fact that $\phi$ is a morphism of algebras.

Berstein's notion of underlying Hopf algebra of a cogroup in $\mathcal{A}s$  \cite{Berstein} is obtained by composing this functor with the forgetful functor from unshuffle bialgebras to classical bialgebras. Proposition \ref{halfcop} unravels why this notion of underlying Hopf algebra of a cogroup could prove in the end instrumental in his work (compare our approach to Berstein's original one).

\section{$B_\infty$--algebras and finite spaces}

The notion of $B_\infty$--algebra was introduced by Getzler and Jones in the category of chain complexes \cite{getzler1994operads}, we consider here the simpler notion of $B_\infty$--algebra in the subcategory $Vect$.

A $B_\infty$--algebra structure on $V$ is, by definition, a Hopf algebra structure on $T(V)$ equipped with the deconcatenation coproduct. That is, an associative algebra structure on $T(V)$  such that the product is a coalgebra map \cite[p. 48]{getzler1994operads}. Since $T(V)$ is cofree as a counital coalgebra in $Vect^+$ for the deconcatenation coproduct, the product map from $T(V)\otimes T(V)$ to $T(V)$ is entirely characterized by its projection to the subspace $V$. This yields another, equivalent, but less tractable and transparent, definition, of $B_\infty$--algebras in terms of structure maps $M_{p,q}:V^{\otimes p}\otimes V^{\otimes q}\longmapsto V,\ p,q\geq 0$ satisfying certain compatibility relations that can be deduced from the associativity of the product -- we refer again to \cite{getzler1994operads} for details.
It is natural to call the cofree Hopf algebra $T(V)$, for $V$ a $B_\infty$--algebra, the $B_\infty$--enveloping algebra of $V$.
The following corollary shows how Theorem \ref{theofund} induces automatically various characterizations of $B_\infty$-enveloping algebras (compare with \cite{Loday}, where the third characterization was obtained).

\begin{cor}\label{structure}
 The following statements are equivalent (as usual all underlying vector spaces belong to $Vect^+$):
 \begin{enumerate}
  \item $H$ is a Hopf algebra, cofree over the space of its primitive elements $V=Prim(H)$.
  \item $H$ is the $B_\infty$-enveloping algebra of a $B_\infty$-algebra $V$.
  \item $H$ is a Hopf algebra and can be equipped with the structure of an infinitesimal bialgebra whose coproduct is the coproduct of $H$.
  \item $H$ is a Hopf algebra and can be equipped with the structure of a shuffle bialgebra whose coproduct is the coproduct of $H$.
 \end{enumerate}

\end{cor}

Let us show now how these ideas apply to finite topologies.

{\bf Notations}. 
Let $X$ be a finite set, and $\T$ be a topology on $X$. For any $Y\subseteq X$, we denote by $\T_{\mid Y}$ the topology induced
by $\T$ on $Y$, that is to say:
$$\T_{\mid Y}=\{O\cap Y\mid O\in \T\}.$$

\begin{defi}\label{def8}
Let $\T\in \TT_n$, $n \geq 1$. 
For $\overline\T\in \SF_n$, the equivalence class of $\T$ in $\SF$, we put:
$$\Delta(\overline\T):=\sum_{O\in \T}\overline{\T_{\mid [n] \setminus O}}\otimes \overline{\T_{\mid O}} \in \F \otimes \F.$$
\end{defi}
We let the reader check that this definition does not depend on the choice of a representative of $\overline\T$ in $\TT$.
The coproduct extends linearly to $\F$, the linear span of finite spaces.

\begin{theo}\label{th9}
\begin{enumerate}
\item $(\F,.,\Delta)$ is a commutative Hopf algebra.
\item $(\F,\succ,\Delta)$ is an infinitesimal bialgebra.
\item $\F$ is the $B_\infty$--enveloping algebra of a $B_\infty$--algebra; more precisely it is a commutative cofree Hopf algebra.
\item It can be equipped with the structure of a shuffle bialgebra or of an unshuffle bialgebra.
\end{enumerate}
\end{theo}

\begin{proof} In the Theorem, all structures are defined in $Vect^+$. 

The last two  assertions follow from Theorem \ref{theofund} together with Corollary \ref{structure}.

Let $\T\in \TT_n$, $n> 0$. 
The coassociativity of $\Delta$ follows from the observations that:
\begin{itemize}
 \item if $O$ is open in $\T$, then the open sets of $O$ are the open sets of $\T$ contained in $O$,
 \item if $O \in \T$ and $O' \in \T_{\mid [n] \setminus O}$, then $O \sqcup O'$ is an open set of $\T$,
 \item if $O_1\subseteq O_2$ are open sets of $\T$, then $O_2\setminus O_1 \in \T_{\mid [n]\setminus O_1}$.
\end{itemize}

We get then:
\begin{eqnarray*}
(\Delta \otimes Id)\circ \Delta(\overline\T)&=&\sum_{O\in \T,\: O'\in \T_{\mid [n]\setminus O} } 
\overline{(\T_{\mid [n]\setminus O})_{\mid ([n]\setminus O)\setminus O'}} \otimes 
\overline{(\T_{\mid [n]\setminus O})_{\mid O'}} \otimes \overline{\T_{\mid O}}\\
&=&\sum_{O\in \T,\: O'\in \T_{\mid [n]\setminus O} } 
\overline{\T_{\mid [n]\setminus (O \sqcup O')}} \otimes \overline{\T_{\mid O'}} \otimes \overline{\T_{\mid O}}.
\end{eqnarray*}
Putting $O_1=O$ and $O_2=O\sqcup O'$:
$$(\Delta \otimes Id)\circ \Delta(\overline\T)=\sum_{O_1\subseteq O_2 \in \T} \overline{\T_{\mid [n]\setminus O_2}}\otimes 
\overline{\T_{\mid O_2 \setminus O_1}}\otimes \overline{\T_{\mid O_1}}=(Id \otimes \Delta)\circ \Delta(\overline\T).$$
This proves that $\Delta$ is coassociative. It is obviously homogeneous of degree $0$. Moreover, $\Delta(1)=1\otimes 1$
and for any $\T\in \TT_n$, $n\geq 1$:
$$\Delta(\overline\T)=\overline\T \otimes 1+1\otimes \overline\T+\sum_{\emptyset \subsetneq O\subsetneq [n]}
\overline{\T_{\mid [n] \setminus O}}\otimes \overline{\T_{\mid O}}.$$
So $\Delta$ has a counit.

Let $\T\in \TT_n$, $\T' \in \TT_{n'}$, $n,n'\geq 0$. By definition of $\T.\T'$:
\begin{eqnarray*}
\Delta(\overline\T.\overline{\T'})&=&\sum_{O\in \T,O'\in \T'} \overline{(\T.\T')_{\mid [n+n']\setminus O.O'}}
\otimes \overline{(\T.\T')_{\mid O.O'}}\\
&=&\sum_{O\in \T,O'\in \T'} \overline{\T_{\mid [n]\setminus O}}.\overline{\T'_{[n']\setminus O'}}
\otimes \overline{\T_{\mid O}}.\overline{\T_{\mid O'}}\\
&=&\sum_{O\in \T,O'\in \T'} \left(\overline{\T_{\mid [n]\setminus O}}
\otimes \overline{\T_{\mid O}}\right).\left(\overline{\T'_{\mid [n']\setminus O'}} \otimes \overline{\T_{\mid O'}}\right)\\
&=&\Delta(\overline\T).\Delta(\overline{\T'}).
\end{eqnarray*}
Hence, $(\SF,.,\Delta)$ is a graded connected commutative Hopf algebra.

By definition of $\T \succ \T'$:
\begin{eqnarray*}
\Delta(\overline\T \succ \overline{\T'})&=&\sum_{O\in \T, O\neq \emptyset}
\overline{(\T\succ \T')_{\mid [n+n']\setminus (O\succ [n'])}}
\otimes \overline{(\T\succ \T')_{\mid O\succ [n']}}\\
&&+\sum_{O'\in \T',O'\neq [n']}
\overline{(\T \succ \T')_{\mid [n+n']\setminus O'(+n)}}
\otimes \overline{(\T\succ \T')_{\mid O'(+n)}}\\
&&+\overline{(\T\succ \T')_{\mid [n+n']\setminus [n'](+n)}}\otimes
\overline{(\T\succ \T')_{ [n'](+n)}}\\
&=&\sum_{O\in \T, O\neq \emptyset}\overline{\T_{\mid  [n]\setminus O}}\otimes
\overline{\T_{\mid O}}\succ \overline{\T'} \\
&&+\sum_{O'\in \T',O'\neq [n']}\overline{\T}\succ \overline{\T'_{\mid [n']\setminus O'}}
\otimes \overline{\T'_{\mid O'}}+\overline\T \otimes \overline{\T'}\\
&=&\sum_{O\in \T, O\neq \emptyset}\left(\overline{\T_{\mid  [n]\setminus O}}\otimes
\overline{\T_{\mid O}}\right)\succ (1\otimes \overline{\T'}) \\
&&+\sum_{O'\in \T',O'\neq [n']}(\overline\T \otimes 1)\succ \left(\overline{\T'_{\mid [n']\setminus O'}}
\otimes \overline{\T'_{\mid O'}}\right)+\overline\T \otimes \overline{\T'}\\
&=&(\Delta(\overline\T)-\overline\T\otimes 1)\succ (1\otimes \overline{\T'})+(\overline\T \otimes 1)\succ (\Delta(\overline\T)-1\otimes \overline{\T'})+\overline\T\otimes \overline{\T'}\\
&=&\Delta(\overline\T)\succ (1\otimes \overline{\T'})+(\overline\T \otimes 1)\succ \Delta(\overline\T)-\overline\T\otimes \overline{\T'}.
\end{eqnarray*}
Hence, $(\SF,\succ,\Delta)$ is an infinitesimal Hopf algebra. \end{proof}

\section{A family of morphisms to quasi-symmetric functions}

\subsection{The Hopf algebra of quasi-symmetric functions}

Let us give some reminders on quasi-symmetric functions. 
Let $A=K[[x_1,x_2,\ldots]]$ be the algebra of commutative formal series in the infinite countable set of indeterminates
$x_n$, $n \geq 1$. A formal series $f\in A$ is quasisymmetric \cite{Gessel, Stanleythesis} if for all strictly increasing maps $f:\N_{>0}\longrightarrow \N_{>0}$,
the coefficients of $x_1^{a_1}\ldots x_n^{a_n}$ and $x_{f(1)}^{a_1}\ldots x_{f(n)}^{a_n}$ in $f$ are equal, for all $a_1,\ldots,a_n \in \N$.
The subalgebra of quasisymmetric formal series is denoted by $\QSym$. For example, if $a=(a_1,\ldots,a_n)$ is a composition,
that is to say a finite sequence of elements of $\N_{>0}$, then the following formal series is quasisymmetric:
$$M_a=\sum_{i_1<\ldots<i_n} x_{i_1}^{a_1}\ldots x_{i_n}^{a_n}.$$
By convention, $M_\emptyset=1$. These elements form a basis of $\QSym$, called the monomial basis. 
Moreover, $\QSym$ is a Hopf algebra \cite{Malvenreut} for the coproduct defined by:
$$\Delta(M_{(a_1,\ldots,a_n)})=\sum_{i=0}^n M_{(a_1,\ldots,a_i)}\otimes M_{(a_{i+1},\ldots,a_n)},$$
for all compositions $(a_1,\ldots,a_n)$.

\subsection{Linear extensions of a finite topology}

In this section, we will write without further comment $\T$ for a representative of the finite space $\overline\T$. 

\begin{defi}
Let $\T$ be a topology on a finite set $E$. 
\begin{enumerate}
\item A linear extension of $\T$ is a map $f:E\longrightarrow \N_{>0}$ such that for all $i,j \in E$,
$$(i\leq_\T j)\Longrightarrow (f(i)\leq f(j)).$$
The set of linear extensions of $\T$ is denoted by $Lin(\T)$.
\item Let $f$ be a linear extension of $\T$.
\begin{enumerate}
\item We shall say that $f$ is standard if $f(E)=[k]$ for a certain integer $k$. 
The set of standard linear extensions of $\T$ is denoted by $Lin_{Std}(\T)$.
\item We denote $f(E)=\{i_1,\ldots,i_k\}$, with $i_1<\ldots<i_k$. We put:
$$P(f)=(|f^{-1}(i_1)|,\ldots,|f^{-1}(i_k)|).$$
Note that $P$ is a map from $Lin(\T)$ to the set of compositions.
\item We put:
$$\alpha(f)=|\{(i,j)\in E\times E\mid i<_\T j \mbox{ and } f(i)=f(j)\}|.$$
Recall that $i<_\T j$ if $i\leq_\T j$ and not $i\sim_\T j$.
Note that $\alpha$ is a map from $Lin(\T)$ to $\N$.
\end{enumerate}\end{enumerate}\end{defi}

{\bf Remarks.} \begin{enumerate}
\item In other words, linear extensions of $\T$ are continuous maps from $E$ to $\N_{>0}$, with the topology induced by the usual total order on $\N_{>0}$.
\item If $\T$ and $\T'$ are homeomorphic, any homeomorphism induces a bijection from $Lin(\T)$ to $Lin(\T')$, which preserves $\alpha$ and $P$.
\item If $f\in Lin(\T)$ and $g:\N_{>0}\longrightarrow \N_{>0}$ is strictly increasing, then $g\circ f \in Lin(\T)$. 
Moreover, $\alpha(g\circ f)=\alpha(f)$ and $P(g\circ f)=P(f)$.
\item For all $f \in Lin(\T)$, there exists a unique $f'\in Lin_{Std}(\T)$, such that there exists a strictly increasing $g:\N_{>0}\longrightarrow \N_{>0}$
with $g\circ f'=f$. This $f'$ is denoted by $Std(f)$. 
\end{enumerate}

\begin{theo}\label{th11}
Let $q \in K$. We put:
$$\phi_q:\left\{\begin{array}{rcl}
\SF&\longrightarrow&\QSym\\
\overline\T&\longrightarrow&\displaystyle \sum_{f\in Lin(\T)}q^{\alpha(f)} \prod_{i\in E(\T)} x_{f(i)},
\end{array}\right.$$
where $E(\T)$ is the set underlying $\T$. This defines a surjective Hopf algebra morphism from $(\F,.,\Delta)$ to $\QSym$. 
Moreover, for all finite spaces $\overline\T$:
$$\phi_q(\overline\T)=\sum_{f \in Lin_{Std}(\T)} q^{\alpha(f)}M_{P(f)}.$$
\end{theo}

\begin{proof} By the first remark above, $\phi_q(\overline\T)$ does not depend on the choice of the representative  $\T$ of $\overline\T$, so $\phi_q(\overline\T)$ is well-defined,
with values in $K[[x_1,x_2,\ldots]]$. By the second remark above, if $\overline\T$ is a finite space:
\begin{align*}
\phi_q(\overline\T)&=\sum_{f \in Lin_{Std}(\T)}q^{\alpha(f)}\sum_{\substack{g:[\max(f)]\longrightarrow \N_{>0},\\ \scriptsize \mbox{strictly increasing}}}
\prod_{i\in E(\T)} x_{g\circ f(1)}\ldots x_{g\circ f(\max(f))}\\
&=\sum_{f \in Lin_{Std}(\T)} q^{\alpha(f)}M_{(|f^{-1}(1)|,\ldots,|f^{-1}(\max(f))|)}\\
&=\sum_{f \in Lin_{Std}(\T)}q^{\alpha(f)} M_{P(f)}.
\end{align*}
So $\phi_q$ takes indeed its values in $\QSym$.\\

Let $\T_1,\T_2$ be representatives of two finite spaces $\overline\T_1,\overline\T_2$ such that $E(\T_1)\cap E(\T_2)=\emptyset$. The set underlying $\T_1.\T_2$ is $E(\T_1)\sqcup E(\T_2)$. 
If $f_i:E(\T_i)\longrightarrow \N_{>0}$ for $i=1,2$, we put:
$$f_1\otimes f_2:\left\{\begin{array}{rcl}
E(\T_1.\T_2)&\longrightarrow&\N_{>0}\\
i&\longrightarrow&\begin{cases}
f_1(i) \mbox{ if }i\in E(\T_1),\\
f_2(i) \mbox{ if }i \in E(\T_2).
\end{cases} \end{array}\right.$$
Then:
$$Lin(\T_1.\T_2)=\{f_1\otimes f_2 \mid (f_1,f_2) \in Lin(\T_1)\times Lin(\T_2)\}.$$
Moreover, $\alpha(f_1\otimes f_2)=\alpha(f_1)+\alpha(f_2)$, as, if $i\leq_{\T_1.\T_2} j$, then $(i,j) \in E(\T_1)^2$ or $(i,j) \in E(\T_2)^2$. We obtain:
\begin{align*}
\phi_q(\overline\T_1.\overline\T_2)&=\sum_{f_1\in Lin(\T_1),f_2 \in Lin(\T_2)} q^{\alpha(f_1)+\alpha(f_2)} \prod_{i\in E(\T_1)\sqcup E(\T_2)}x_{f_1\otimes f_2(i)}\\
&=\sum_{f_1\in Lin(\T_1),f_2 \in Lin(\T_2)} q^{\alpha(f_1)+\alpha(f_2)} \prod_{i\in E(\T_1)} x_{f_1(i)}\prod_{i\in E(\T_2)} x_{f_2(i)}\\
&=\left(\sum_{f_1 \in Lin(\T_1)} q^{\alpha(f_1)} \prod_{i\in E(\T_1)}x_{f_1(i)}\right) \left(\sum_{f_2 \in Lin(\T_2)} q^{\alpha(f_2)} \prod_{i\in E(\T_2)}x_{f_2(i)}\right) \\
&=\phi_q(\overline\T_1)\phi_q(\overline\T_2).
\end{align*}
This shows that $\phi_q$ is multiplicative.

\ \par

Let $\overline\T$ be a finite space. We put:
\begin{align*}
A&=\{(I,f_1,f_2)\mid I\mbox{ open set of }\T, f_1 \in Lin_{Std}(\T_{\mid E(T)-I}), f_2 \in Lin_{Std}(\T_{\mid I})\},\\
B&=\{(f,k)\mid f \in Lin_{Std}(\T), 0\leq k\leq \max(f)\}.
\end{align*}
We put:
$$F:\left\{\begin{array}{rcl}
B&\longrightarrow&A\\
(f,k)&\longrightarrow&(f^{-1}(\{k+1,\ldots,\max(f)\}), Std(f_{\mid [k]}), Std(f_{\mid\{k+1,\ldots,\max(f)\}})).
\end{array}\right.$$
This is well-defined: we put $F(f,k)=(I,f_1,f_2)$.
\begin{itemize}
\item Let $i \in I$ and $j\geq_\T i$. Then $f(i) \geq k+1$. As $f \in Lin(\T)$, $f(j)\geq f(i)$, so $f(j)\geq k+1$ and $j\in I$: $I$ is an open set of $\T$.
\item By restriction, $f_1$ is a linear extension of $\T_{\mid E(\T)-I}$ and $f_2$ is a linear extension of $\T_{\mid I}$.
\end{itemize}
Moreover, $F$ is injective: if $F(f,k)=F(g,l)=(I,f_1,f_2)$, then $k=l=\max(f_1)$. As $f$ is standard, for all $i \in E(\T)$:
\begin{itemize}
\item if $i\notin I$, $f(i)=g(i)=f_1(i)$,
\item if $i\in I$, $f(i)=g(i)=f_2(i)+k$.
\end{itemize} 
Finally, $F$ is surjective: if $(I,f_1,f_2) \in A$, let $f:E(\T)\longrightarrow \N_{>0}$, defined by:
\begin{itemize}
\item if $i\notin I$, $f(i)=f_1(i)$,
\item if $i\in I$, $f(i)=f_2(i)+\max(f_1)$.
\end{itemize} 
Let us prove that $f \in Lin(\T)$. If $i\leq_\T j$ in $E(\T)$, then:
\begin{itemize}
\item If $i\in I$, as $I$ is an open set of $\T$, $j\in I$. As $f_2 \in Lin(\T_{\mid I})$, then $f_2(i)\leq f_2(j)$, so $f(i)\leq f(j)$.
\item If $i\notin I$ and $j \in I$, then $f(i)\leq k<f(j)$.
\item If $i,j\notin I$, as $f_1 \in Lin(\T_{\mid E(\T)-I})$, $f(i)=f_1(i) \leq f_1(j)=f(j)$. 
\end{itemize}
$f$ is clearly standard, and $F(f,\max(f_1))=(I,f_1,f_2)$. 

As a conclusion, $F$ is bijective. Moreover, if $F(f,k)=(I,f_1,f_2)$, as if $i\in I$ and $j\notin I$, $f(i)\neq f(j)$, 
then $\alpha(f)=\alpha(f_1)+\alpha(f_2)$, and $P(f)$ is the concatenation of $P(f_1)$ and $P(f_2)$. So:
\begin{align*}
(\phi_q \otimes \phi_q)\circ \Delta(\T)&=\sum_{(I,f_1,f_2)\in A} q^{\alpha(f_1)+\alpha(f_2)} M_{P(f_1)}\otimes M_{P(f_2)}\\
&=\sum_{(f,k) \in B} q^{\alpha(f)} M_{(|f^{-1}(1)|,\ldots, |f^{-1}(k)|)}\otimes M_{(|f^{-1}(k+1)|,\ldots, |f^{-1}(\max(f))|)}\\
&=\sum_{f\in Lin_{Std}(f)} q^{\alpha(f)} \Delta(M_{P(f)})\\
&=\Delta \circ \phi_q(\T).
\end{align*}
So $\phi_q$ is a Hopf algebra morphism.\\

Let $(a_1,\ldots,a_k)$ be a composition. Let $\T$ be the topology on a set $A_1\sqcup \ldots \sqcup A_k$, with $|A_i|=a_i$ for all $i$,
defined by $x\leq_\T y$ if, and only if, $x\in A_i$ and $y\in A_j$, with $i\leq j$. Then:
$$\phi_q(\overline{\T})=M_{(a_1,\ldots,a_k)}+R,$$ where $R$ is in the linear span of the $M_b$, with $length(b)<k$.
By a triangularity argument, $\phi_q$ is surjective. \end{proof}

{\bf Examples.} Let $a,b,c \geq 1$.
\begin{align*}
\phi_q(\tdun{$a$})&=M_{(a)},\\
\phi_q(\tddeux{$a$}{$b$})&=M_{(a,b)}+q^{ab}M_{(a+b)},\\
\phi_q(\tdun{$a$}\tdun{$b$})&=M_{(a,b)}+M_{(b,a)}+M_{(a+b)},\\
\phi_q(\tdtroisdeux{$a$}{$b$}{$c$})&=M_{(a,b,c)}+q^{ab}M_{(a+b,c)}+q^{bc}M_{(a,b+c)}+q^{ab+ac+bc}M_{(a+b+c)},\\
\phi_q(\tdtroisun{$a$}{$c$}{$b$})&=M_{(a,b,c)}+M_{(a,c,b)}+M_{(a,b+c)}+q^{ab}M_{(a+b,c)}\\
&+q^{ac}M_{(a+c,b)}+q^{ab+ac}M_{(a+b+c)},\\
\phi_q(\pdtroisun{$c$}{$a$}{$b$})&=M_{(a,b,c)}+M_{(b,a,c)}+M_{(a+b,c)}+q^{ac}M_{(b,a+c)}\\
&+q^{bc}M_{(a,b+c)}+q^{ac+bc}M_{(a+b+c)},\\
\phi_q(\tddeux{$a$}{$b$}\tdun{$c$})&=M_{(a,b,c)}+M_{(a,c,b)}+M_{(c,a,b)}+M_{(a,b+c)}\\
&+M_{(a+c,b)}+q^{ab}M_{(a+b,c)}+q^{ab}M_{(c,a+b)}+q^{ab}M_{(a+b+c)},\\
\phi_q(\tdun{$a$}\tdun{$b$}\tdun{$c$})&=M_{(a,b,c)}+M_{(a,c,b)}+M_{(b,a,c)}+M_{(b,c,a)}+M_{(c,a,b)}+M_{(c,b,a)}\\
&+M_{(a+b,c)}+M_{(a+c,b)}+M_{(b+c,a)}\\
&+M_{(a,b+c)}+M_{(b,a+c)}+M_{(c,a+b)}+M_{(a+b+c)}.\\
\end{align*}

\begin{prop}
We define a product $\succ_q$ on $\QSym$ by:
$$M_{(a_1,\ldots,a_k)}\succ_q M_{(b_1,\ldots,b_l)}=M_{(a_1,\ldots,a_k,b_1,\ldots,b_l)}+q^{a_kb_1}M_{(a_1,\ldots,a_{k-1},a_k+b_1,b_2,\ldots,b_l)}.$$
Then $(\QSym,\succ_q,\Delta)$ is an infinitesimal bialgebra and $\phi_q$ is a morphism of infinitesimal bialgebras from $(\mathcal{F},\succ,\Delta)$
to $(\QSym,\succ_q,\Delta)$.
\end{prop}

\begin{proof}
Let $\overline\T_1,\overline\T_2$ be two nonempty finite spaces in $\TT_n$, resp. $\TT_m$. Let us prove that $\phi_q(\overline\T_1\succ \overline\T_2)=\phi_q(\overline\T_1)\succ_q \phi_q(\overline\T_2)$.
We choose standard representatives $\T_1,\T_2$.
Let $f \in Lin(\T_1 \succ \T_2)$. We put $f_1=f_{\mid [n]}$ and $f_2=f_{\mid \{n+1,...,n+m\}}$. If $i\in [n]$ and $j\in \{n+1,...,n+m\}$,
then $i\leq_{\T_1\succ \T_2} j$, so $f(i)\leq f(j)$. Hence, $\max(f_1)\leq \min(f_2)$. We then define:
\begin{align*}
A_<&=\{f\in Lin_{Std}(\T_1\succ \T_2)\mid \max(f_1)<\min(f_2)\},\\
A_=&=\{f\in Lin_{Std}(\T_1\succ \T_2)\mid \max(f_1)=\min(f_2)\}.
\end{align*}
We deduce from the preceding remark that $Lin_{Std}(\T_1\succ \T_2)=A_<\sqcup A_=$. Let us consider the maps
$F_<:Lin_{Std}(\T_1)\times Lin_{Std}(\T_2)\longrightarrow A_<$ and
$F_=:Lin_{Std}(\T_1)\times Lin_{Std}(\T_2)\longrightarrow A_=$ defined by:
\begin{align*}
F_<(f_1,f_2)&:\left\{\begin{array}{rcl}
[n+m]&\longrightarrow& \N_{>0}\\
i&\longrightarrow&\begin{cases}
f_1(i) \mbox{ if }i\leq n,\\
f_2(i-n)+\max(f_1)\mbox{ if } i>n;\\
\end{cases}
\end{array}\right.\\
F_=(f_1,f_2)&:\left\{\begin{array}{rcl}
[n+m]&\longrightarrow& \N_{>0}\\
i&\longrightarrow&\begin{cases}
f_1(i) \mbox{ if }i\leq n,\\
f_2(i-n)+\max(f_1)-1\mbox{ if } i>n.
\end{cases}
\end{array}\right. \end{align*}
Both are clearly bijections. Moreover, if $(f_1,f_2) \in Lin_{Std}(\T_1)\times Lin_{Std}(\T_2)$:
\begin{itemize}
\item $\alpha(F_<(f_1,f_2))=\alpha(f_1)+\alpha(f_2)$ and:
$$\alpha(F_=(f_1,f_2))=\alpha(f_1)+\alpha(f_2)+|f_1^{-1}(\max(f_1))||f_2^{-1}(\min(f_2))|,$$
the last term corresponding to the pairs $(i,j) \in [n]\times \{n+1,...,n+m\}$, with $f_1(i)=\max(f_1)$ and $f_2(j-n)=\min(f_2)$, as for such a pair $(i,j)$,
$f(i)=f(j)$ and $i<_{\T_1\succ \T_2} j$.
\item If $P(f_1)=(a_1,\ldots,a_k)$ and $P(f_2)=(b_1,\ldots,b_l)$, then $P(F_<(f_1,f_2))=(a_1,\ldots,a_k,b_1,\ldots,b_l)$ and
$P(F_=(f_1,f_2))=(a_1,\ldots,a_k+b_1,\ldots,b_l)$, so:
$$M_{P(F_<(f_1,f_2))}+q^{a_kb_1}M_{P(F_=(f_1,f_2))}=M_{P(f_1)}\succ_q M_{P(f_2)}.$$
\end{itemize}
This gives:
\begin{align*}
\phi_q(\overline\T_1\succ \overline\T_2)&=\sum_{f \in A_<} q^{\alpha(f)} M_{P(f)}+\sum_{f \in A_=} q^{\alpha(f)} M_{P(f)}\\
&=\sum_{(f_1,f_2) \in Lin_{Std}(\T_1)\times Lin_{Std}(\T_2)}q^{\alpha(f_1)+\alpha(f_2)}M_{P(f_1)}\succ_q M_{P(f_2)}\\
&=\phi_q(\overline\T_1)\succ_q \phi_q(\overline\T_2).
\end{align*}
As $\phi_q$ is surjective and $(\mathcal{F},\succ,\Delta)$ is an infinitesimal bialgebra,we obtain that $(\QSym,\succ_q,\Delta)$ is also
an infinitesimal bialgebra. \end{proof}

{\bf Remark.} Theorem 4.1 of \cite{AguiarSottile} gives an interpretation of the pair $(\QSym,\zeta_\QSym)$ as a final object in the category of graded,
connected Hopf algebras together with a character, where $\zeta_\QSym$ is the character of $\QSym$ defined by
$\zeta_\QSym(M_{(a_1,\ldots,a_k)})=\delta_{k,1}$ for all composition $(a_1,\ldots,a_k)$ of length $k\geq 1$. With this formalism, $\phi_q$ is the Hopf algebra
morphism in this category associated to the character $\zeta_q=\zeta_\QSym \circ \phi_q$ of $\F$. For any finite space $\overline\T$, of degree $n$,
$$\zeta_q(\overline\T)=q^{\alpha((n))}=q^{|\{(i,j) \in E(\T)\mid i<_\T j\}|}.$$
In particular, $\zeta_1(\overline\T)=1$ for all $\overline\T$; and for $q=0$,
$$\zeta_0(\overline\T)=\begin{cases}
1\mbox{ if $\T=\tdun{$a_1$}\:\ldots \tdun{$a_k$}\: $  for a certain $(a_1,\ldots,a_k) $},\\
0\mbox{ otherwise}.
\end{cases}$$

{\bf Remark.}  The map $\phi_q$ is not injective. For example, if $\overline\T$ and $\overline\T'$ are the following two finite spaces:
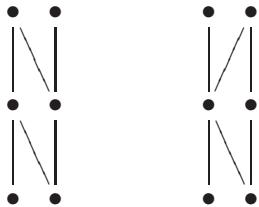
\begin{displaymath}
\xymatrix@C=4pt{ \bullet \ar@{-}[d] \ar@{-}[dr]  & \bullet \ar@{-}[d] &&&&& \bullet \ar@{-}[d]& \bullet \ar@{-}[d] \ar@{-}[dl] \\ 
		\bullet \ar@{-}[d] \ar@{-}[dr] &  \bullet \ar@{-}[d]&&&&&  \bullet \ar@{-}[d] \ar@{-}[dr]  & \bullet \ar@{-}[d]  \\ 
		\bullet &\bullet &&&&&\bullet &  \bullet } 
\end{displaymath}
then $\overline\T\neq \overline\T'$ but $\phi_q(\overline\T)=\phi_q(\overline\T')$ (this is a linear span of 204 terms). 
However, it is possible to prove that if $\T$ and $\T'$ are two topologies on the same set $E$, they are equal if, and only if, $Lin(\T)=Lin(\T')$.

\bibliographystyle{amsplain}
\bibliography{biblio}

\end{document}